\documentclass[11pt, a4paper]{amsart}
\usepackage{cgeometry}
\newcommand\blfootnote[1]{%
  \begingroup
  \renewcommand\thefootnote{}\footnote{#1}%
  \addtocounter{footnote}{-1}%
  \endgroup
}

\bibliography{ComplexGeometry}
\title{Analytic Bertini theorem}
\author{Mingchen Xia}

\begin{document}

\maketitle

\begin{abstract}
We prove an analytic Bertini theorem, generalizing a previous result of Fujino and Matsumura.
\end{abstract}


 \normalsize 

\section{Introduction}
Let $X$ be a connected complex projective manifold of dimension $n\geq 1$. Given any base-point free linear system $\Lambda$ on $X$, it follows from the classical Bertini theorem \cite{Jou80} that a general hyperplane $H$ of $\Lambda$ is smooth. 
Let $\varphi$ be a quasi-plurisubharmonic (quasi-psh) function on $X$. For a general member $H\in \Lambda$, the multiplier ideal sheaf $\mathcal{I}(\varphi|_H)$ makes sense.
It is natural to wonder if 
\begin{equation}\label{eq:IresH}
\mathcal{I}(\varphi|_H)=\mathcal{I}(\varphi)|_H
\end{equation}
holds for general $H$. It is well-known that the $\mathcal{I}(\varphi|_H)\subseteq \mathcal{I}(\varphi)|_H$ direction always holds for a general $H$, as a consequence of the Ohsawa--Takegoshi $L^2$-extension theorem. 
Conversely, it is easy to construct examples such that the set $\mathcal{B}$ of $H\in \Lambda$ where the equality fails is not contained in any proper Zariski closed subset of $\Lambda$. A natural question arises: is the set $\mathcal{B}$ small in a suitable sense? This kind of problem was first studied by Fujino and Matsumura, see \cite{FM21} and \cite{Fuj18}. They proved that the complement of $\mathcal{B}$ is dense with respect to the complex topology of $\Lambda$ (regarded as a projective space). More recently, Meng--Zhou \cite{MZ21} proved that the complement of $\mathcal{B}$ has zero Lebesgue measure.
In this paper, we prove the following refinement:
\begin{theorem}\label{thm:main}
There is a pluripolar set $\Sigma\subseteq \Lambda$ such that for all $H\in \Lambda\setminus \Sigma$, $H$ is smooth and \eqref{eq:IresH} holds. 
\end{theorem}
This result affirmatively answers a problem of Boucksom, see \cite[Question~1.2]{Fuj18}. From the point of view of pluripotential theory, this theorem is quite natural: a small set in pluripotential theory just means a pluripolar set. As shown in \cite[Example~3.12]{FM21}, the exceptional set is not contained in a countable union of proper Zariski closed subsets in general, so \cref{thm:main} seems to be the optimal result. We also prove a more general analytic Bertini type result for fibrations \cref{cor:fibra}.

Let us mention a key advantage of \cref{thm:main}: our theorem can be applied to a countable family of quasi-psh functions at the same time, see \cref{cor:Irestunif}. This corollary makes it possible to perform induction on the dimension when studying psh singularities. 

\textbf{Acknowledgements}. I would like to thank Osamu Fujino and Sébastien Boucksom for comments on the draft.

I am indebted to the referees for various helpful remarks. In particular, one referee pointed out that equality holds in \cref{lma:pull-backmis}, which simplifies our proof of the main theorem.

\section{Analytic Bertini theorem}

In this section, varieties or algebraic varieties mean reduced separated schemes of finite type over $\mathbb{C}$.

\begin{definition}
Let $Y$ be a complex projective manifold. A subset $A\subseteq Y$ is 
\begin{enumerate}
    \item \emph{co-pluripolar} if $Y\setminus A$ is pluripolar. When $\dim Y=1$, we also say $A\subseteq Y$ is co-polar.
    \item \emph{co-meager} if $Y\setminus A$ is contained in a countable union of proper Zariski closed sets.
\end{enumerate}
We say a condition in $y\in Y$ is satisfied \emph{quasi-everywhere} if there is a co-pluripolar subset $Y_0\subseteq Y$ such that the condition is satisfied for $y\in Y_0$.
\end{definition}
Clearly, a co-meager set is co-pluripolar. Both classes are preserved by countable intersections.

\begin{lemma}\label{lma:pull-backmis}
Let $\pi:Y\rightarrow X$ be a smooth morphism of smooth algebraic varieties. Let $\varphi$ be a quasi-plurisubharmonic function on $X$, then 
\begin{equation}\label{eq:missmp}
\pi^*\mathcal{I}(\varphi)= \mathcal{I}(\pi^*\varphi)\,.
\end{equation}
\end{lemma}
Here $\mathcal{I}(\varphi)$ denotes the multiplier ideal sheaf of $\varphi$ in the sense of Nadel. Observe that as $\pi$ is flat, $\pi^*\mathcal{I}(\varphi)$ is a subsheaf of $\mathcal{O}_Y$, so in \eqref{eq:missmp} equality makes sense, the two sheaves are actually equal, not just isomorphic.

\begin{proof}
As pointed out by the referee, $\pi^*\mathcal{I}(\varphi)\supseteq  \mathcal{I}(\pi^*\varphi)$ is proved in \cite[Proposition~14.3]{Dem12}. So it suffices to prove the reverse inclusion.

By decomposing $\pi$ into the composition of an étale morphism and a projection locally, it suffices to deal with the two cases separately. Fix a local section $f$ of $\mathcal{I}(\varphi)$.

Assume that $\pi:X\times \mathbb{C}^n\rightarrow X$ is the natural projection. Fix a volume form $\mathrm{d}V$ on $X$.
Take the product volume form $\mathrm{d}V\otimes \mathrm{d}\lambda$ on $X\times \mathbb{C}^n$, where $\mathrm{d}\lambda$ denotes the Lebesgue measure. It follows from Fubini theorem that $|\pi^*f|^2e^{-\pi^*\varphi}$ is locally integrable with respect to $\mathrm{d}V\otimes \mathrm{d}\lambda$.

Now assume that $\pi$ is étale. The change of variable formula shows that $|\pi^*f|^2e^{-\pi^*\varphi}$ is locally integrable.
\end{proof}

In \cref{lma:pull-backmis},
we do not really need the algebraic structures on $X$ and $Y$. For general complex manifolds, it suffices to apply the co-area formula.

We recall the notion of positive metrics on a torsion-free coherent sheaf.
\begin{definition}
Let $X$ be a smooth complex algebraic variety. Let $\mathcal{F}$ be a torsion-free (algebraic) coherent sheaf on $X$. Let $Z\subseteq X$ be the smallest Zariski closed set such that $\mathcal{F}|_{X\setminus Z}=\mathcal{O}_{X\setminus Z}(F)$ for some vector bundle $F$ on $X\setminus Z$. A \emph{singular Hermitian metric} (resp. \emph{positive singular Hermitian metric}) on $\mathcal{F}$ is a singular Hermitian metric (resp. Griffiths positively curved singular Hermitian metric) on $F$ in the sense of \cite{Rau15}. 
\end{definition}

\begin{theorem}\label{thm:Bert}
Let $X$ be a connected projective manifold of dimension $n\geq 1$. Let $\varphi$ be a quasi-plurisubharmonic function on $X$. Let $p:X\rightarrow \mathbb{P}^N$ be a morphism ($N\geq 1$). Define
\[
\mathcal{G}:=\left\{\,H\in |\mathcal{O}_{\mathbb{P}^N}(1)|: H':=H\cap X \text{ is smooth and } \mathcal{I}(\varphi|_{H'})=\mathcal{I}(\varphi)|_{H'}\,\right\}\,.
\]
Then $\mathcal{G}\subseteq |\mathcal{O}_{\mathbb{P}^N}(1)|$ is co-pluripolar.
\end{theorem}
\begin{remark}
Here and in the sequel, we slightly abuse the notation by writing $H\cap X$ for $p^{-1}H$, the scheme-theoretic inverse image of $H$. In other words, $H\cap X:=H\times_{\mathbb{P}^N} X$.

By definition, any $H\in |\mathcal{O}_{\mathbb{P}^N}(1)|$ such that $p^{-1}H=\emptyset$ lies in $\mathcal{G}$.
\end{remark}

We briefly sketch the argument. We need to show that for quasi-every $H\in |\mathcal{O}_{\mathbb{P}^N}(1)|$, the restriction formula \eqref{eq:IresH} holds. A standard argument in algebraic geometry allows us to reduce the proof of \eqref{eq:IresH} to proving the corresponding equality on global sections, after tensoring with a sufficiently ample line bundle $L$. In this case, we group the pairs consisting of $H\in \Lambda$ and points on $H\cap X$ as a single fibration $\pi_1:U\rightarrow \Lambda$. The theory of positivity of direct images allows us to construct a coherent sheaf $\mathcal{F}$ on $\Lambda$ endowed with a positive metric $h_{\mathcal{H}}$ out of $\pi_1$ and $\varphi$. By the construction of $h_{\mathcal{H}}$, the locus where the restriction formula \eqref{eq:IresH} fails is contained in the singular locus of $h_{\mathcal{H}}$, which finishes the proof.

\begin{proof}

Take an ample line bundle $L$ with a smooth Hermitian metric $h$ such that $c_1(L,h)+\ddc\varphi\geq 0$, where $c_1(L,h)$ is the first Chern form of $(L,h)$, namely the curvature form of $h$. 
Let $\mathcal{L}$ be the invertible sheaf corresponding to $L$.
We introduce $\Lambda:=|\mathcal{O}_{\mathbb{P}^N}(1)|$ to simplify our notations.

\begin{step}\label{ste:1}
\end{step}
We prove that the following set is co-pluripolar:
\[
\begin{split}
\mathcal{G}_{\mathcal{L}}:=\left\{H\in \Lambda: H\cap X \text{ is smooth and } H^0\left(H\cap X,\omega_{H\cap X}\otimes \mathcal{L}|_{H\cap X}\otimes \mathcal{I}(\varphi|_{H\cap X})\right)= \right.\\
\left. H^0\left(H\cap X,\omega_{H\cap X}\otimes \mathcal{L}|_{H\cap X}\otimes \mathcal{I}(\varphi)|_{H\cap X}\right)\right\}\,.
\end{split}
\]
Here $\omega_{H\cap X}$ denotes the dualizing sheaf of $H\cap X$.

Let $U\subseteq \Lambda\times X$ be the closed subvariety whose $\mathbb{C}$-points correspond to pairs $(H,x)\in \Lambda\times X$ with $p(x)\in H$.  
Let $\pi_1:U\rightarrow \Lambda$ be the natural projection. We may assume that $\pi_1$ is surjective, as otherwise there is nothing to prove.

Observe that $U$ is a local complete intersection scheme by \emph{Krulls Hauptidealsatz} and \emph{a fortiori} a Cohen--Macaulay scheme. It follows from miracle flatness \cite[Theorem~23.1]{Mat89} that the natural projection $\pi_2:U\rightarrow X$ is flat. As the fibers of $\pi_2$ over closed points of $X$ are isomorphic to $\mathbb{P}^{N-1}$, it follows that $\pi_2$ is smooth. Thus, $U$ is smooth as well.

In the following, we will construct pluripolar sets $\Sigma_1\subseteq \Sigma_2 \subseteq \Sigma_3\subseteq \Sigma_4\subseteq \Lambda$ such that the behaviour of $\pi_1$ is improved successively on the complement of $\Sigma_i$.

\textbf{Step~\ref*{ste:1}.1}
The usual Bertini theorem shows that there is a proper Zariski closed set $\Sigma_1\subseteq \Lambda$ such that $\pi_1$ has smooth fibres outside $\Sigma_1$. This is slightly more general than the version that one finds in \cite{Har}, see \cite[Théorème~6.3]{Jou80} for a proof.

\textbf{Step~\ref*{ste:1}.2}
By Koll\'ar's torsion-free theorem \cite[Theorem~C]{FM21}, 
\[
\mathcal{F}^i:=R^i\pi_{1*}\left(\omega_{U/\Lambda}\otimes \pi_2^*\mathcal{L}\otimes \mathcal{I}(\pi_2^*\varphi) \right)
\]
is torsion-free for all $i$. Here $\omega_{U/\Lambda}$ denotes the relative dualizing sheaf of the morphism $U\rightarrow \Lambda$.
Thus, there is a proper Zariski closed set $\Sigma_2\subseteq \Lambda$ such that 
\begin{enumerate}
    \item $\Sigma_2\supseteq \Sigma_1$.
    \item The $\mathcal{F}^i$'s are locally free outside $\Sigma_2$.
    \item $\omega_{U/\Lambda}\otimes \pi_2^*\mathcal{L}\otimes \mathcal{I}(\pi_2^*\varphi)$ is $\pi_1$-flat on $\pi_1^{-1}(\Lambda\setminus \Sigma_2)$ \cite[Théorème~6.9.1]{EGAIV-2}. 
\end{enumerate}
We write $\mathcal{F}=\mathcal{F}^0$.
 By cohomology and base change \cite[Theorem~III.12.11]{Har}, for any $H\in \Lambda\setminus \Sigma_2$,
 the fibre $\mathcal{F}|_H$ of $\mathcal{F}$ is given by 
 \[
\mathcal{F}|_H= H^0\left(\pi_{1,H},\omega_{U/\Lambda}|_{\pi_{1,H}}\otimes \pi_2^*\mathcal{L}|_{\pi_{1,H}}\otimes \mathcal{I}(\pi_2^*\varphi)|_{\pi_{1,H}} \right)\,.
\]
Here $\pi_{1,H}$ denotes the fibre of $\pi_1$ at $H$.

\textbf{Step~\ref*{ste:1}.3} In order to proceed, we need to make use of the Hodge metric $h_{\mathcal{H}}$ on $\mathcal{F}$ defined in \cite{HPS18}. We briefly recall its definition in our setting. 
By \cite[Section~22]{HPS18}, we can find a proper Zariski closed set $\Sigma_3\subseteq \Lambda$ such that
\begin{enumerate}
    \item $\Sigma_3\supseteq \Sigma_2$.
    \item $\pi_1$ is submersive outside $\Sigma_3$.
    \item Both $\mathcal{F}$ and $\pi_{1*}\left(\omega_{U/\Lambda}\otimes \pi_2^*\mathcal{L} \right)/\mathcal{F}$  are locally free outside $\Sigma_3$.
    \item For each $i$, 
    \[
    R^i\pi_{1*}\left(\omega_{U/\Lambda}\otimes \pi_2^*\mathcal{L} \right)
    \]
    is locally free outside $\Sigma_3$.
\end{enumerate}
Then for any $H\in \Lambda\setminus \Sigma_3$,
\[
H^0(H\cap X,\omega_{H\cap X}\otimes \mathcal{L}|_{H\cap X}\otimes \mathcal{I}(\varphi|_{H\cap X}))\subseteq \mathcal{F}|_H\subseteq H^0(H\cap X,\omega_{H\cap X}\otimes \mathcal{L}|_{H\cap X})\,.
\]
See \cite[Lemma~22.1]{HPS18}.

Now we can give the definition of the Hodge metric on $\Lambda\setminus \Sigma_3$.
Given any $H\in \Lambda\setminus \Sigma_3$, any $\alpha\in \mathcal{F}|_H$, the Hodge metric is defined as
\[
h_{\mathcal{H}}(\alpha,\alpha):=\int_{X\cap H} |\alpha|^2_{he^{-\varphi}|_{X\cap H}}\in [0,\infty]\,.
\]
Observe that $h_{\mathcal{H}}(\alpha,\alpha)<\infty$ if and only if $\alpha\in H^0(H\cap X,\omega_{H\cap X}\otimes \mathcal{L}|_{H\cap X}\otimes \mathcal{I}(\varphi|_{H\cap X}))$. Moreover, $h_{\mathcal{H}}(\alpha,\alpha)>0$ if $\alpha\neq 0$.
It is shown in  \cite{HPS18} (c.f. \cite[Theorem~3.3.5]{PT18}) that $h_{\mathcal{H}}$ is indeed a singular Hermitian metric and it extends to a positive metric on $\mathcal{F}$.

\textbf{Step~\ref*{ste:1}.4}.
The determinant $\det h_{\mathcal{H}}$ is singular at all $H\in \Lambda\setminus \Sigma_3$ such that 
\[
H^0(H\cap X,\omega_{H\cap X}\otimes \mathcal{L}|_{H\cap X}\otimes \mathcal{I}(\varphi|_{H\cap X}))\neq \mathcal{F}|_H\,.
\]
As the map $\pi_2$ is smooth, we have  $\pi_2^*\mathcal{I}(\varphi)= \mathcal{I}(\pi_2^*\varphi)$ by \cref{lma:pull-backmis}. 
Under the identification $\pi_{1,H}\cong H\cap X$, we have
\[
\pi_2^*\mathcal{I}(\varphi)|_{\pi_{1,H}}\cong \mathcal{I}(\varphi)|_{H\cap X}\,.
\]
Thus we have the following inclusions:
\[
H^0(H\cap X,\omega_{H\cap X}\otimes \mathcal{L}|_{H\cap X}\otimes \mathcal{I}(\varphi|_{H\cap X}))\subseteq H^0(H\cap X,\omega_{H\cap X}\otimes \mathcal{L}|_{H\cap X}\otimes \mathcal{I}(\varphi)|_{H\cap X})= \mathcal{F}|_H\,.
\]
Recall that the first inclusion follows from the Ohsawa--Takegoshi $L^2$-extension theorem.
Hence $\det h_{\mathcal{H}}$ is singular at all $H\in |\mathcal{O}_{\mathbb{P}^N}(1)|\setminus \Sigma_3$ such that
\[
H^0(H\cap X,\omega_{H\cap X}\otimes \mathcal{L}|_{H\cap X}\otimes \mathcal{I}(\varphi|_{H\cap X}))\neq H^0(H\cap X,\omega_{H\cap X}\otimes \mathcal{L}|_{H\cap X}\otimes \mathcal{I}(\varphi)|_{X\cap H})\,.
\]
Let $\Sigma_4$ be the union of $\Sigma_3$ and the set of all such $H$.
Since the Hodge metric $h_{\mathcal{H}}$ is positive (\cite[Theorem~3.3.5]{PT18} and \cite[Theorem~21.1]{HPS18}),
its determinant $\det h_{\mathcal{H}}$ is also positive (\cite[Proposition~1.3]{Rau15} and \cite[Proposition~25.1]{HPS18}), it follows that $\Sigma_4$ is pluripolar. As a consequence, $\mathcal{G}_{\mathcal{L}}$ is co-pluripolar.

\begin{step}\label{ste:2}
\end{step}

Fix an ample invertible sheaf $\mathcal{S}$ on $X$. The same result holds with $\mathcal{L}\otimes \mathcal{S}^{\otimes a}$ in place of $\mathcal{L}$. Thus the set
\[
A:=\bigcap_{a=0}^{\infty}\mathcal{G}_{\mathcal{L}\otimes \mathcal{S}^{\otimes a}}
\]
is co-pluripolar.
For each $H\in W$ such that $X\cap H$ is smooth and $\mathcal{I}(\varphi|_{X\cap H})\neq \mathcal{I}(\varphi)|_{X\cap H}$, let $\mathcal{K}$ be the following cokernel:
\[
0\rightarrow \mathcal{I}(\varphi|_{X\cap H})\rightarrow \mathcal{I}(\varphi)|_{X\cap H}\rightarrow \mathcal{K}\rightarrow 0\,.
\]
By Serre vanishing theorem, taking $a$ large enough, we may guarantee that
\[
H^1(X\cap H,\omega_{X\cap H}\otimes (\mathcal{L}\otimes \mathcal{S}^{\otimes a})|_{X\cap H}\otimes \mathcal{I}(\varphi|_{X\cap H}))=0
\]
and
\[
H^0(X\cap H,\omega_{X\cap H}\otimes (\mathcal{L}\otimes \mathcal{S}^{\otimes a})|_{X\cap H}\otimes \mathcal{K})\neq 0\,.
\]
Then
\[
H^0(X\cap H,\omega_{X\cap H}\otimes(\mathcal{L}\otimes \mathcal{S}^{\otimes a})|_{X\cap H}\otimes\mathcal{I}(\varphi|_{X\cap H}))\neq H^0(X\cap H,\omega_{X\cap H}\otimes(\mathcal{L}\otimes \mathcal{S}^{\otimes a})|_{X\cap H}\otimes\mathcal{I}(\varphi)|_{X\cap H})\,.
\]
Thus $H\not\in A$. We conclude that $\mathcal{G}$ is co-pluripolar.
\end{proof}

\begin{remark}
As pointed out by the referee, in \cite{FM21}, Fujino and Matsumura also treated the case when $X$ is not projective. It is of interest to understand if \cref{thm:Bert} can be extended to non-projective complex manifolds as well.
\end{remark}

Note that the argument for $\pi:U\rightarrow \Lambda$ in the proof of \cref{thm:Bert} works for more general fibrations.
With essentially the same proof, we can similarly prove an analytic Bertini type theorem for fibrations. 
\begin{corollary}\label{cor:fibra}
Let $\pi:U\rightarrow W$ be a surjective morphism of projective varieties. Let $(L,\phi)$ be a Hermitian pseudo-effective line bundle on $U$, namely $L$ is a holomorphic line bundle on $U$ and $\phi$ is a plurisubharmonic metric on $L$. Then there is a pluripolar subset $\Sigma\subseteq W$ such that for all $w\in W\setminus \Sigma$, $U_w:=\pi^{-1}(w)$ is smooth and we have $\mathcal{I}(\phi|_{U_w})=\mathcal{I}(\phi)|_{U_w}$.
\end{corollary}

\begin{corollary}\label{cor:Irestunif}
Let $X$ be a projective manifold of pure dimension $n\geq 1$.  Let $\Lambda$ be a base-point free linear system. Let $\varphi$ be a quasi-psh function on $X$.
Then there is a pluripolar subset $\Sigma\subseteq \Lambda$ such that for any $H\in \Lambda\setminus \Sigma$ and any real number $k>0$,
\begin{equation}\label{eq:Ikphi}
\mathcal{I}(k\varphi|_H)=\mathcal{I}(k\varphi)|_H
\end{equation}
and we have a short exact sequence for all $k>0$,
\begin{equation}\label{eq:Ikphires}
0\rightarrow \mathcal{I}(k\varphi)\otimes \mathcal{O}_X(-H)\rightarrow \mathcal{I}(k\varphi)\rightarrow \mathcal{I}(k\varphi|_H)\rightarrow 0\,.
\end{equation}
\end{corollary}
\begin{proof}
First observe that by the strong openness theorem \cite{GZopen} in order to verify \eqref{eq:Ikphi} for all real $k>0$, it suffices to verify it for $k$ lying in a countable subset $K\subseteq \mathbb{R}_{>0}$.

Applying \cref{thm:Bert} to each $k\varphi$ with $k\in K$ and each connected component of $X$, we find that there is a pluripolar set $\Sigma_1\subseteq \Lambda$ such that for any $H\in\Lambda\setminus \Sigma_1$ and any $k\in K$, \eqref{eq:Ikphi} holds. On the other hand, the union of the sets of associated primes of $\mathcal{I}(k\varphi)$ for $k>0$ is a countable set, hence the set $A$ of $H\in \Lambda$ that avoids them is co-meager. It suffices to take $\Sigma=\Sigma_1\cup (\Lambda\setminus A)$.
\end{proof}
Following the terminology of \cite{DX22}, given quasi-psh functions $\varphi$ and $\psi$ on $X$, we say $\varphi\sim_{\mathcal{I}}\psi$ if for all real $k>0$, $\mathcal{I}(k\varphi)=\mathcal{I}(k\psi)$.
\begin{corollary}
Let $X$ be a projective manifold of pure dimension $n\geq 1$.
Let $\varphi$, $\psi$ be quasi-psh functions on $X$ such that $\varphi\sim_{\mathcal{I}}\psi$. Let $\Lambda$ be a base-point free linear system. Then there is a pluripolar subset $\Sigma\subseteq \Lambda$ such that for any $H\in \Lambda\setminus \Sigma$, $\varphi|_H$ and $\psi|_H$ are both quasi-psh functions on $H$ and we have $\varphi|_H\sim_{\mathcal{I}} \psi|_H$.
\end{corollary}

\section*{Competing Interests}
I declare that I have no competing interests related to this paper.


\printbibliography

\bigskip
  \footnotesize

  Mingchen Xia, \textsc{Department of Mathematics, Chalmers Tekniska Högskola, Göteborg}\par\nopagebreak
  \textit{Email address}, \texttt{xiam@chalmers.se}\par\nopagebreak
  \textit{Homepage}, \url{http://www.math.chalmers.se/~xiam/}.

\end{document}